\documentclass[a4paper, 10pt]{amsart}
\usepackage{amsfonts,mathrsfs,amsmath,amscd,amssymb}
\usepackage[pdftex]{graphicx}
\usepackage{hyperref}
\usepackage{color}
\newtheorem{theorem}{Theorem}

\usepackage[english]{babel}
\usepackage[all,cmtip]{xy}
\usepackage{xcolor}
\newtheorem{proposition}[theorem]{\sf Proposition}

\newtheorem{definition}[theorem]{\sf Definition}

\newtheorem{remark}[theorem]{\sf Remark}

\def\C{\mathbb C}
\def \Z{\mathbb Z}
\def \N{\mathbb N}

\def \O{\mathcal O}

\def \ra{\rightarrow}

\let\olddefinition\definition
\renewcommand{\definition}{\olddefinition\normalfont}

\let\oldremark\remark
\renewcommand{\remark}{\oldremark\normalfont}

\begin{document}
%%%%%%%%%%%%%%%%%%%%
\author{Tanya Kaushal Srivastava}
\title[Hilbert scheme of Supersingular Enriques]{Pathologies of Hilbert scheme of points of supersingular Enriques surface}
\date{\today}
\begin{abstract}
We show that Hilbert schemes of points on supersingular Enriques surface in characteristic 2, $Hilb^n(X)$, for $n \geq 2$ are simply connected, symplectic varieties but are not irreducible symplectic as the hodge number $h^{2,0} > 1$, even though a supersingular Enriques surface is an irreducible symplectic variety. These are the classes of varieties which appear only in characteristic 2 and they show that the hodge number formula for G\"ottsche-Soergel does not hold over characteristic 2. It also gives examples of varieties with trivial canonical class which are neither irreducible symplectic nor Calabi-Yau, thereby showing that there are strictly more classes of simply connected varieties with trivial canonical class in characteristic 2 than over $\C$ as given by Beauville-Bogolomov decomposition theorem.
\end{abstract}
\thanks{I would like to thank M. Zdanwociz for various mathematical discussions which lead to this article and Tamas Hausel for hosting me in his research group at IST Austria.  This research has received funding from the European Union's Horizon 2020 research and innovation programme under Marie Sklodowska-Curie Grant Agreement No. 754411.}
\maketitle

\smallskip
\noindent \textbf{Keywords.} Hilbert scheme of points, Enriques surface, Positive characteristic, Lifting to characteristic zero, Irreducible Symplectic varieties.\\
\noindent \textbf{2010 Mathematics Subject Classification:} 14G17, 14J28, 14J10. 

%%%%%%%%%%%%%%%%%%%%%%%%%%%%%%%%%%%%%%%%%%%%%%%%
\section*{Introduction}

This article stems from the exploration of the remark of F. Charles in \cite{FC-ISV} that the naive definition of irreducible symplectic varieties in positive characteristic is not satisfactory. We give a concrete example displaying one not very desirable phenomenon of ``\textbf{not} being deformation closed" class. The class of examples we study occurs only in characteristic 2 but we expect to find similar classes of examples in higher dimensions once more of birational classification of  varieties is studied for varieties in positive characteristic (cf. Remark \ref{charneq2}). 

The naive definition of (irreducible) symplectic varieties in arbitrary characteristic is given as:
\begin{definition}\cite[Definition 3.1]{FuLi} \label{ISV}
A smooth projective variety $X$ over an algebraically closed field $k$ is said to be \textbf{symplectic} if its \'etale fundamental group is zero and it has a nowhere vanishing global two form $\sigma \in H^0(X, \Omega^2)$. A symplectic variety is said to be \textbf{irreducible symplectic} if it is a symplectic variety and the nowhere vanishing global two form $\sigma$ generates $H^0(X, \Omega^2)$, i.e., $\dim H^0(X, \Omega^2) =1 $.   
\end{definition} 
The existence of a nowhere vanishing global two form $\sigma$ implies that the canonical bundle is trivial, and in case char $k \neq 2$ that the dimension of $X$ is even. This is because in characteristic 2, odd dimensional skew symmetric 2 forms do not have to be singular.  

\textbf{Question:} Does there exist odd dimensional irreducible symplectic varieties over a field of characteristic 2? 

\begin{remark}
Over $\C$, analytification of a smooth projective variety gives us a K\"ahler manifold and then being irreducible symplectic is equivalent to being a Hyperk\"ahler manifold. Indeed, first note that \'etale fundamental group being trivial is equivalent to topological fundamental group of the analytification is zero, as   
$$
\pi_{et}(X)=\widehat{\pi_{top}(X^{an})},
$$ 
where $\hat{-}$ denotes profinite completion, see \cite[Remark I.5.1 (c)]{Milne80}. Thus, $\pi_{top}(X^{an}) = 0 \Rightarrow \pi_{et}(X) =0$. For the other side, as our varieties have $c_1(X) = 0$, the Beauville-Bogomolov theorem \cite{BB} implies that we have 
$$ 
0 \ra  \Z^k \ra \pi_{top}(X) \ra G \ra 0,
$$
where $G$ is a finite group. Since profinite completion is right exact and $\hat{G} = G$, see \cite{RZ},  we have
$$
\widehat{\Z^k} \ra \pi_{et}(X) \ra G \ra 0.
$$
Now, $\pi_{et}(X) = 0 \Rightarrow G=0$ which implies $\pi_{top}(X)= \Z^k$ but taking profinite completion implies that $k =0$. Thus $\pi_{top}(X) = 0$. Now the equivalence of the two definitions can be seen using Yau's proof of Calabi theorem and Beauville-Bogomolov decomposition, see \cite{HuyCHM} for details. 
\end{remark}

Moreover over $\C$, $c_1(X) =0$ is  a deformation invariant, as well as being simply connected and having a  non-degenerate two form. The latter two results are easy to see using Ehresmann's fibration theorem. 

The main result of this article is
\begin{proposition} [Proposition \ref{SSEsurface} , \ref{fund}, \ref{HN}]
Let $X$ be a supersingular Enriques surface over $k$, algebraically closed field of characteristic $2$, then it is an irreducible symplectic variety. Moreover, $Hilb^n(X)$ for $n \geq 2$ is a symplectic variety which is not irreducible symplectic. 
\end{proposition}

The article is organized as follows: In section 1, we recall basic facts abouts Enriques Surface and note that it in characteristic 2, supersingular Enriques surface are indeed irreducible symplectic varieties as defined above in Definition \ref{ISV} and in section 2, we begin by recalling the properties of Hilbert scheme of points on a surface and show that in case of supersingular Enriques surface, the Hilbert scheme has the desired properties as claimed above. 

%%%%%%%%%%%%%%%%%%%%%%%%%%%%%%%%%%%%%%%%%%%%%%%%%%%%%%555
\section{Enriques Surface}

We refer to the classification of surfaces by Bomberi-Mumford \cite{BMIII} for details.

\begin{definition}
An \textbf{Enriques surface} is a smooth and proper surface $X$ of finite type over an algebraically closed field $k$ of characteristic $p \geq 0$ such that 
$$
\omega_X \sim_{num} \O_X \ \text{and} \ b_2(X) =10,
$$
where $\sim_{num}$ denotes numerical equivalence and $b_i$ denotes the $i$-th \'etale or crystalline Betti number. 
\end{definition}

In char $k \neq 2$, these surfaces have the following invariants
\[\begin{matrix}
\omega_X \not\cong \O_X & \omega_X^2 \cong \O_X &  p_g= h^0(X, \omega_X)=0 & h^{0,1}= 0\\
\chi(\O_X) = 1 & c_2 = 12 & b_1 = 0 & b_2 = 10.
\end{matrix}\]
And the Hodge diamond is (char $\neq 2$):
\[\begin{matrix}
 & & h^{2,2} &  &\\
 & h^{1,2}&  &h^{2,1}  &\\
h^{0,2} & & h^{1,1} &  &h^{2,0} \\
 & h^{0,1}&  &h^{1,0}  & \\
 & & h^{0,0} & &
\end{matrix}
\ := \
\begin{matrix}
 & & 1 &  &\\
 & 0&  &0  &\\
0 & & 10 &  &0 \\
 & 0&  &0  & \\
 & & 1 & &
\end{matrix}\]
As $\omega_X^2 \cong \O_X$, this defines a defines an \'etale double cover $\tilde{X} \ra X$, where $\tilde{X}$ is a K3 surface. The existence of a K3 double cover implies $\pi_{et}(X) = \Z/2\Z$ as $\pi_{et}(\tilde{X})= 0$. 

In case the characteristic of the ground field is 2, we have the following invariants of Enriques surfaces:
\[\begin{matrix}
\omega_X \sim_{num} \O_X & \chi(\O_X) = 1 & c_2 = 12 & b_1 = 0 & b_2 = 10.
\end{matrix}\]
However, in char k = 2, hodge symmetry can fail, i.e., $h^{0,1} = 1$, while $h^{1,0}= 0$. Moreover, $\omega_X \not\cong \O_X$ if and only if $h^{0,1}=0$.  In case $h^{0,1}=1$, we can study the action of Frobenius on $H^1(X, \O_X)$, which is either zero or bijective, this leads to the following definition and characterization of Enriques surface in characteristic 2:
\begin{definition}
An Enriques surface $X$ is called 
\begin{enumerate}
\item \textbf{Classical} if $h^1(\O_X) = 0$, hence $\omega_X \not\cong \O_X$ and $\omega^2 \cong \O_X$.
\item \textbf{1-ordinary/Singular} if $h^1(\O_X) = 1$, hence $\omega_X \cong \O_X$ and Frobenius is bijective on $h^1(\O_X)$.
\item \textbf{Supersingular} if  $h^1(\O_X) = 1$, hence $\omega_X \cong \O_X$ and Frobenius is zero on $h^1(\O_X)$.
\end{enumerate}
\end{definition}

The following theorem gives us the fundamental group of the three kinds of Enriques surface:
\begin{theorem}[Bomberi-Mumford]
Let $X$ be an Enriques surface with char $k = 2$. Then $X$ admits a canonical covering space $\pi: \tilde{X} \ra X$ of degree $2$, whose structure group is
\begin{enumerate}
\item $\mu_2$ if $X$ is classical $\Rightarrow \pi_{et}(X) = 0$
\item $\Z/2\Z$ if $X$ is 1-ordinary/singular $\Rightarrow \pi_{et}(X) = \Z/2\Z$
\item  $\alpha_2$ if $X$ is supersingular $\Rightarrow \pi_{et}(X) = 0$.
\end{enumerate}
\end{theorem}

Now we are ready to note that supersingular Enriques surfaces are irreducible symplectic varieties.

\begin{proposition} \label{SSEsurface}
Supersingular Enriques surface is an irreducible symplectic variety over an algebraically closed field of char $p =2$ and it does not lift to any irreducible symplectic variety in char 0, even more it lifts to a variety which has $\omega_X \not\cong \O_X$.   
\end{proposition}

\begin{proof}
In \cite[Theorem 4.10]{Lie15}, Liedtke has shown that every Enriques surface lifts to characteristic zero, see remark \ref{Enriqueslift} below. \\
Now, all we are have to do is note that for supersingular Enriques surface, we have $\omega_X \cong \O_X$, which is equivalent to having $H^0(X, \Omega_X^2) = k \sigma$, where $\sigma$ is nowhere degenerate 2 form (since we are in $\dim 2$), and $\pi_{et}(X)=0$, thus we are done. 
\end{proof}

\begin{remark} \label{Enriqueslift}
Note that in  \cite[Theorem 4.10]{Lie15}  Liedtke  constructs the lift of a supersingular Enriques surface only as an algebraic space over a possibly ramified extension of W(k), the Witt ring of $k$, but proves all such lifts are projective, which implies that they are indeed just schemes, it follows from  \cite[II.6.16]{KnutsonAS}. Moreover, the generic fibers of these lifts are Enriques surfaces over a field of characteristic 0, indeed, use the results of \cite[Theorem 11.4]{LiedtkeCAS} or \cite[Section 9]{K-U95} to see that the Kodaira dimension of the generic fiber is zero, Euler characteristic is 1 and first betti number is 0; then classification of surfaces as in \cite{LiedtkeCAS} allows us to conclude.
\end{remark}

Also, one can deform supersingular Enriques surface to a classical one, thus unlike in characteristic zero, the class of varieties with trivial canonical bundle is not deformation closed in characteristic $p > 0$, \cite{AZST}. 
%%%%%%%%%%%%%%%%%%%%%%%%%%%%%%%%%%%%%%%%%%%%%%
\section{Hilbert scheme of Enriques surface}
In this section we will prove that the Hilbert scheme of points on a supersingular Enriques surface is a symplectic variety but not irreducible symplectic by showing that it is simply connected, has a non-degenerate two form but the group $H^0(\Omega^2)$ is not generated by it. We start by recalling a few basic results on Hilbert scheme of points on a surface, which immediately implies that the Hilbert scheme of points on supersingular Enriques has trivial canonical bundle and then we compute its fundamental group and $h^{2,0}$.   For basic results and proofs about Hilbert scheme of points of a scheme (in arbitrary characteristic)  see \cite[Chapter 7]{BK}.

\subsection{Hilbert schemes of points on a surface}\label{HilbS}

Let $X$ denote a nonsingular projective surface over an algebraically closed field $k$ of characteristic $p \geq 0$ and $Hilb^n(X)$ denote the \textbf{Hilbert scheme of length $n$} on $X$. It is the unique scheme representing the functor 
\begin{eqnarray*}
\textbf{Hilb}^n(X): \{Sch / k \} &\xrightarrow{\hspace*{2.5cm}} \{Sets\} \\
S &\mapsto \{ \pi: Z \ra S \ \text{flat}; \ Z \subset X_S \\ 
 &\text{length $n$ closed subscheme}\}.
\end{eqnarray*} 
Moreover, it is a smooth projective variety of dimension of dim $2n$, see  \cite[Theorem 7.2.3, 7.4.1]{BK}. Note that any point of $Hilb^n(X)$ is a closed subscheme $Z$ of $X$ such that $h^0(\O_Z)=n$ and generic points of $Hilb^n(X)$ are reduced, that is, they are the closed subschemes of $n$-distinct points of $X$. \\
Let $X^{(n)}$ denote the \textbf{$n$-th symmetric product} of $X$, which is defined to be the quotient of $X^n$ by the natural action of $S_n$ which permutes the factors of $X^n$, $X^n/S_n$. It is a projective scheme \cite[Lemma 7.1.1]{BK}. Any point of $X^{(n)}$ can be written as a formal sum $\sum_i m_i p_i$, where $\sum_i m_i =n$, all $m _i \in \N$ and $p_i \in X$ are distinct points.  There is a natural morphism of schemes from the Hilbert scheme of $n$ points to the $n$-th symmetric product called \textbf{Hilbert-Chow morphism}, which at the level of points can be described as follows: 
\begin{eqnarray*}
Hilb^n(X) &\xrightarrow{\hspace*{0.7cm} \gamma \hspace*{0.6cm}}  X^{(n)}\\
Z &\mapsto \sum_{p \in X} l(\O_{Z,p})p, 
\end{eqnarray*}
where $l(-)$ denotes the length of a module and the sum is a formal sum. This morphism is projective and gives a crepant resolution of $X^{(n)}$  \cite[Theorem7.3.1, 7.4.6]{BK}. Thus, we have $\omega_{Hilb^n(X)} \cong \omega_{X^{(n)}}$ and \cite[Lemma 7.1.7]{BK} implies that $\omega_{Hilb^n(X)} \cong (\omega_X)^{(n)}$. Thus, if $X$ is a variety (of dim $\geq 2$) with trivial canonical bundle, then so is $Hilb^n(X)$. 

\subsection{The Fundamental group of Hilbert scheme on supersingular Enriques}

We will compute the fundamental group by lifting the Hilbert scheme of supersingular Enriques to characteristic zero and computing the fundamental group of the geometric generic fiber and then using  specialization of fundamental groups to bound the type of elements  in the fundamental group of Hilbert scheme of supersingular Enriques. Then, since the Hilbert scheme is unirational so the fundamental group has no 2-torsion but the specialization map on fundamental groups shows that the fundamental group can be either a 2-torsion group or zero, so we have that the fundamental group is zero.

As a supersingular Enriques surface, $X$, lifts to characteristic 0 (over a DVR $A$ with residue field $k$ and fraction field $K$) as a (relative) Enriques surface $X_A$, with generic fiber $X_K$ (see Remark \ref{Enriqueslift}), we can construct lifts of $Hilb^n(X)$ as $Hilb^n(X_A)$ with generic fiber as $Hilb^n(X_K)$, the Hilbert scheme of points on an Enrique surface over $K$, a field of characteristic zero. Now using the specialization morphism of \'etale fundamental groups \cite[Corollary X.2.3]{SGA1} we have the following surjective map:
\begin{equation} \label{surjection}
\pi_{et}(Hilb^n(X_{\bar{K}})) \ra \pi_{et}(Hilb^n(X)) \ra 0
\end{equation}
and from \cite[Expose XIII, Proposition 4.6]{SGA1}, we have $$ 
\pi_{et}(Hilb^n(X_{\bar{K}})) \cong \pi_{et}(Hilb^n(X_{\C})),
$$ 
where $X_{\bar{K}}$ is the geometric generic fiber of the lift $X_A$ of $X$ and $X_{\C}$ is its base change to $\C$. We can always do base change to $\C$ since we are working with varieties of finite type over a field.  

The analytic fundamental group of $Hilb^n(X_{\C})$ can be computed as follows: (for example using \cite[Proposition 2.1]{OGrady13}, we get following isomorphisms for $n \geq 2$) 
$$
\Z/2\Z \cong \pi_{an}(X_{\C})^{ab} \cong H_1(X_{\C}, \Z) \cong \pi_{an}(Hilb^n(X_{\C})). 
$$
And since the groups are finite, we have that $\pi_{et}(X_{\C}) \cong \Z/2\Z$, see \cite[Remark I.5.1(c)]{Milne80}.
Thus it is just a 2-torsion group and then equation \ref{surjection} implies that $\pi_{et}(Hilb^n(X))$ is a group of order at most 2.  

On the other hand, supersingular Enrqiues surfaces are weakly unirational as proved by Blass in \cite{Bla} and this implies that $Hilb^n(X)$ is also weakly unirational, as we can construct a generically surjective rational map from $\mathbb{P}^n$ as follows:  
$$
\mathbb{P}^{2n} \dashrightarrow X^{n} \ra X^{(n)} \sim_{bir} Hilb^n(X).  
$$
In \cite{Eke}(see also \cite{CL-A03}), Ekedahl has proven that fundamental group of a weakly unirational variety over an algebraically closed field of characteristic $p >0$ is a finite group of order prime to $p$.

Combining the above observations we have 
\begin{proposition}\label{fund}
The Hilbert scheme of n points on supersingular Enriques surface, $Hilb^n(X)$ is simply connected.   
\end{proposition}

\subsection{Computation of $H^0(Hilb^n(X), \Omega_{Hilb^n(X)}^2)$}
The construction of a no-where vanishing 2-form on $Hilb^n(X)$, where $X$ is a supersingular surface, using the no-where vanishing 2-form on $X$ follows verbatim as in the case of varieties over $\C$ in \cite[Section 2.1.2]{OGrady13}. The only thing that needs to be checked is that $X^{(n)}$ is still $\mathbb{Q}$-factorial, this is \cite[Lemma 7.1.9]{BK}. 

Now we show that $h^{2,0}(Hilb^n(X)) > 1$. Recall that the hodge numbers 
$$
h^{q,0} = \dim_k H^0(X, \Omega^q)
$$
 are birational invariants (see for example, \cite[Ex. II.8.8]{HartAG}). Thus in our case (see section \ref{HilbS}) we have 
\begin{equation} \label{birHN}
h^{2,0}(Hib^n(X)) = h^{2,0}(X^{(n)}).
\end{equation}
Next note that since $ \Omega_{X^n} \cong \oplus_ip_i^*(\Omega_X)$, where $p_i: X^n \ra X$ is the i-th projection from the $n$-product of $X$ to $X$, and we have $\Omega^q_{X^{(n)}}= (\Omega^q_{X^n})^{S_n}$ for $q \geq 0$ , on the level of global sections we get 
$$
h^{q,0}(X^{(n)}) = \Gamma(X, \Omega^q_{X^n})^{S_n}.
$$
In case $q= 2$, we have 
$$
\Omega^2_{X^n} = \oplus_{\{i,j; i < j\}} p_i^*\Omega_X \otimes p_j^*\Omega_X \oplus \wedge^2(p_i^*\Omega_X).
$$
To write it explicitly in the case of $n=2$, we get 
\begin{eqnarray*}
\Omega^2_{X^2} =& p_1^*\Omega_X\otimes p_2^*\Omega_X \oplus \wedge^2p_1^*\Omega_X \oplus \wedge^2p_2^*\Omega_X \\
= &p_1^*\Omega_X\otimes p_2^*\Omega_X \oplus p_1^*\wedge^2\Omega_X \oplus p_2^*\wedge^2\Omega_X.
\end{eqnarray*}
The action of $S_n$ on $\Omega^2_{X^n}$ permutes the corresponding factors of $\Omega_X$. One of the collections of invariant sections of $\Omega_{X^n}^2$ are given by 
$$
\oplus \lambda p_i^*\sigma,
$$
where $\sigma \in H^0(X, \Omega^2_X)$ is a section, and $\lambda \in k^*$.  Next, note that the sections of $\oplus_{\{i,j; i < j\}} p_i^*\Omega_X \otimes p_j^*\Omega_X$ say given as $\oplus_{\{i,j; i < j\}} p_i^*\alpha_i \otimes p_j^*\alpha_j$, under the action of $S_n$ change to  $\oplus_{\{i',j'; i' < j'\}}p_{i'}^*\alpha_{i'} \otimes p_{j'}^*\alpha_{j'} \oplus _{\{i',j'; i' > j'\}}(-1)p_{j'}^*\alpha_{j'} \otimes p_{i'}^*\alpha_{i'}$, where the $i'$ (resp. $j'$) is the image of $i$ (resp. $j$) under the action of a permutation of $S_n$. \\

Explicitly in the case of $n=2$, we can see that the only non-trivial permutation $(12)$ of $S_2$ sends $p_1^*\alpha_1 \otimes p_2^*\alpha_2$ to $(-1)p_1^*\alpha_1 \otimes p_2^*\alpha_2$. \\

Thus, in case the characteristic of the ground field $k$ is not equal to 2, we see that these parts do not contribute to any more $S_n$ invariants, however in case characteristic $k $ is equal to $2$, which implies $-1 = 1$, we see that we will have more $S_n$ invariants of the form    $\oplus_{\{i,j; i < j\}} p_i^*\alpha_i \otimes p_j^*\alpha_j$ with $\alpha_i = \alpha_j$.   To put it formally we have shown that 

\begin{proposition} \label{HN}
For $X$ a supersingular Enriques surface defined over an algebraically closed field of characteristic 2, we have $h^{2,0}(Hilb^n(X)) > 1$.
\end{proposition}

Moreover, the above arguments also shows the reason why G\"ottsche-Soergel Hodge number formula \cite[Theorem 2.3.14(3)]{GottscheBook}, proved in \cite{Got-Sor}, fails in characteristic 2, we miss out on a lot of invariant sections. \\

Thus, the Hilbert scheme of points on a supersingular Enriques surface gives examples of varieties with trivial canonical class which are symplectic variety but neither irreducible symplectic nor Calabi-Yau, thereby showing that there are strictly more classes of simply connected varieties with trivial canonical class in characteristic 2 than as  over $\C$ given by Beauville-Bogomolov decomposition theorem. There is also a Beauville Bogomolov type partial decomposition theorem in positive characteristic by Patakfalvi-Zdanowicz  \cite{ZMBB} but it is only for 1-ordinary varieties, which is equivalent to being $F$-split \cite[Proposition 2.13]{ZMBB} and $Hilb^n(X)$ is 1-ordinary (resp. F-split) if $X$ is so. However in this case, i.e., Hilbert scheme of 1-ordinary Enriques surface is not a symplectic variety. One of the better places to look for the definition of irreducible symplectic varieties would be to find out if we can identify more decomposition classes in the Patakfalvi-Zdanowicz version of Beauville-Bogomolov theorem.

\begin{remark} \label{charneq2}
 The class of examples we study occurs only in characteristic 2. Being in characteristic 2, influenced the example in two ways, it gave us supersingular Enriques as irreducible symplectic varieties, even though they are not so in any other characteristic and it lead to the failure of  G\"ottsche-Soergel's hodge number counting formula, thereby preventing the Hilbert scheme of supersingular Enriques surface from being irreducible symplectic.    
Since the failure of the hodge number formula is not to be expected in any other characteristic and there are no further irreducible symplectic surfaces, we do not expect to have further examples of this type. 
\end{remark}
%%%%%%%%%%%%%%%%%%%

\end{document}